\documentclass[11pt]{amsart}
\usepackage[leqno]{amsmath}
\usepackage{amssymb,latexsym,soul,cite,amsthm,color,enumitem,graphicx,mathtools,microtype,accents}
\usepackage[colorlinks=true,urlcolor=crimson,citecolor=crimson,linkcolor=crimson,linktocpage,pdfpagelabels,bookmarksnumbered,bookmarksopen]{hyperref}
\definecolor{crimson}{rgb}{0.86, 0.08, 0.24}
\usepackage[english]{babel}
\usepackage[left=2.7cm,right=2.7cm,top=2.4cm,bottom=2.5cm]{geometry}

\numberwithin{equation}{section}

\newtheorem{theorem}{Theorem}[section]
\theoremstyle{plain}
\newtheorem{lemma}[theorem]{Lemma}
\theoremstyle{plain}
\newtheorem{proposition}[theorem]{Proposition}
\theoremstyle{plain}

\newtheorem{definition}[theorem]{Definition}
\theoremstyle{definition}
\newtheorem{remark}[theorem]{Remark}

\newcommand{\N}{{\mathbb N}}

\newcommand{\R}{{\mathbb R}}
\newcommand{\eps}{\varepsilon}
\newcommand{\beq}{\begin{equation}}
\newcommand{\eeq}{\end{equation}}
\renewcommand{\le}{\leqslant}
\renewcommand{\ge}{\geqslant}

\newcommand{\w}{W^{s,p}_0(\Omega)}
\newcommand{\fpl}{(-\Delta)_p^s\,}
\newcommand{\cs}{C^0_s(\overline\Omega)}
\newcommand{\ds}{{\rm d}_\Omega^s}
\newcommand{\p}{p^*_s}

\makeatletter
\newcommand{\leqnomode}{\tagsleft@true}
\newcommand{\reqnomode}{\tagsleft@false}
\makeatother

\newenvironment{enumroman}{\begin{enumerate}

}{\end{enumerate}}

\title[Sublinear fractional $p$-Laplacian]{On a doubly sublinear fractional $p$-Laplacian equation}

\author[A.\ Iannizzotto, S.\ Mosconi]{Antonio Iannizzotto, Sunra Mosconi}

\address[A.\ Iannizzotto]{Dipartimento di Matematica e Informatica
\newline\indent
Universit\`a degli Studi di Cagliari
\newline\indent
Via Ospedale 72, 09124 Cagliari, Italy}
\email{antonio.iannizzotto@unica.it}

\address[S.\ Mosconi]{Dipartimento di Matematica e Informatica
\newline\indent
Universit\`a degli Studi di Catania
\newline\indent
Viale A.\ Doria 6, 95125 Catania, Italy}
\email{mosconi@dmi.unict.it}


\begin{document}

\begin{abstract}
We prove a bifurcation result for a Dirichlet problem driven by the fractional $p$-Laplacian (either degenerate or singular), in which the reaction is the difference between two sublinear powers of the unknown. In our argument, a fundamental role is played by a Sobolev vs.\ H\"older minima principle, already known for the degenerate case, which here we extend to the singular case with a simpler proof.
\end{abstract}

\maketitle

\begin{center}
Version of \today\
\end{center}

\section{Introduction and main result}\label{sec1}

\noindent
The present note is mainly devoted to the study of the following Dirichlet type problem for a nonlinear, nonlocal elliptic equation:
\[(P_\lambda) \qquad \begin{cases}
\fpl u = \lambda u^{q-1}-u^{r-1} & \text{in $\Omega$} \\
u > 0 & \text{in $\Omega$} \\
u = 0 & \text{in $\R^N\setminus \Omega$.}
\end{cases}\]
Here $\Omega\subset\R^N$ ($N\ge 2$) is a bounded domain with a $C^{1,1}$-smooth boundary, $p>1$, $s\in(0,1)$ are s.t.\ $ps<N$, and the leading operator is the $s$-fractional $p$-Laplacian, namely, the gradient of the functional
\[u \mapsto \frac{1}{p}\iint_{\R^N\times\R^N}\frac{|u(x)-u(y)|^p}{|x-y|^{N+ps}}\,dx\,dy,\]
defined in a convenient fractional Sobolev space $\w$. If $p\ge 2$ and $u$ is smooth enough, such operator admits the pointwise representation
\[\fpl u(x) = 2\lim_{\eps\to 0^+}\int_{B_\eps^c(x)}\frac{|u(x)-u(y)|^{p-1}(u(x)-u(y))}{|x-y|^{N+ps}}\,dy.\]
The reaction involves two $(p-1)$-sublinear powers of $u$ with exponents $1<r<q<p$, and depends on a parameter $\lambda>0$. 
\vskip2pt
\noindent
The literature on boundary value problems driven by the fractional $p$-Laplacian includes a wide variety of results, starting from \cite{ILPS} and ensuring the existence, multiplicity, or uniqueness of weak solutions via variational and/or topological methods, see for instance \cite{BS,CMSS,DPQ,FI,IL,IM1,IMP,PSY}. In particular, \cite{IMP} focuses on the logistic equation, involving a reaction similar to that of $(P_\lambda)$, but with $1<q<r$ and $r>p$.
\vskip2pt
\noindent
Our main result states that problem $(P_\lambda)$ may admit two, one, or no solutions, depending on the value of $\lambda$, with precise order relations between solutions and convergence at a threshold value of the parameter. For a precise definition of solution we refer to Section \ref{sec2}.

\begin{theorem}\label{pqr}
Let $\Omega\subset\R^N$ ($N\ge 2$) be a bounded domain with a $C^{1,1}$-smooth boundary, $1<r<q<p$, $s\in(0,1)$ s.t.\ $ps<N$. Then, there exists $\lambda_*>0$ s.t.\
\begin{enumroman}
\item\label{pqr1} for all $0<\lambda<\lambda_*$ $(P_\lambda)$ has no solution;
\item\label{pqr2} $(P_{\lambda_*})$ has at least one solution $u_{\lambda_*}$;
\item\label{pqr3} for all $\lambda>\lambda_*$, $(P_\lambda)$ has a biggest solution $u_\lambda$ and at least another solution $v_\lambda$ s.t.\ $u_\lambda>v_\lambda$ in $\Omega$;
\item\label{pqr4} the map $\lambda\mapsto u_\lambda$ is increasing, i.e., if $\lambda>\mu\ge \lambda^*$   then $u_\lambda>u_\mu$ in $\Omega$.
\end{enumroman}
\end{theorem}

\noindent
In Theorem \ref{pqr} above, which can be seen as a bifurcation result for problem $(P_\lambda)$, a notable feature is represented by the strict inequalities between solutions in \ref{pqr3}, \ref{pqr4}, which are obtained by applying the mountain pass theorem to a suitably truncated energy functional. Also, we emphasize the strong monotonicity of the mapping $\lambda\mapsto u_\lambda$ in \ref{pqr4}. We will actually prove stronger comparisons between solutions, namely $u_\lambda-v_\lambda\in {\rm int}(\cs_+)$ (see \eqref{intcs} below for this space), and similarly $u_\lambda-u_\mu\in {\rm int}(\cs_+)$. Surprisingly enough, the situation for $(P_\lambda)$ reflects that of the super-diffusive logistic equation studied in \cite{IMP} ($p<q<r$), despite in our case the powers on the right-hand side are below the homogeneity of the leading operator, which is $p-1$.
\vskip2pt
\noindent
The inequalities in Theorem \ref{pqr} are obtained by considering several truncations of the reaction, and applying either direct minimization, or min-max schemes to the corresponding energy functionals. This in turn requires some general properties of solutions to fractional elliptic equations. First, we employ general {\em strong minimum/comparison principles} from \cite{IMP}, which in fact ensure not only strict inequalities but even a fractional Hopf-type property at boundary points. For instance, nontrivial non-negative solutions in $\Omega$ of the equation
\[\fpl u = \lambda u^{q-1}-u^{r-1}\]
not only fulfil $u>0$ in $\Omega$, but also
\[\inf_{\Omega}\frac{u}{\ds} > 0,\]
where ${\rm d}_\Omega$ denotes the distance from the boundary of $\Omega$. Denoting $\nu$ the inward normal unit vector of $\partial\Omega$, the inequality above in turn implies that for all $x\in\partial\Omega$
\beq
\label{hopf}
\frac{\partial u}{\partial\nu^s}(x) := \lim_{t\to 0^+}\frac{u(x+t\nu)}{t^s} > 0.
\eeq
It is notable that this a purely nonlocal feature. Indeed, in the case of the local $p$-Laplacian (i.e., for $s=1$ formally) the strong minimum principle requires conditions on the reaction (see \cite{V}), which rule out $(p-1)$-sublinear terms. In the absence of such conditions the problem may admit infinite nontrivial, non-negative solutions with compact support in $\Omega$, but proving the existence of a strictly positive one is troublesome. This is not the case in the fractional setting, as the nonlocality of the operator $\fpl$ allows to prove a strong minimum principle also for $(p-1)$-sublinear reactions, see Proposition \ref{smp}. Similarly, in the strong comparison principle similar conditions on the reactions can be dropped in the nonlocal framework (compare \cite{GV} and Proposition \ref{scp}). 
\vskip2pt
\noindent
Also, the use of truncations requires some care in handling the function spaces. Indeed, the energy functional $\Phi_\lambda$ for problem $(P_\lambda)$ is naturally defined in the fractional Sobolev space $\w$, but when truncating the reaction, we essentially restrict it to certain order-related subsets with empty interior (think for instance of the positive order cone of $\w$), thus losing information about the local minimizers of $\Phi_\lambda$. This problem was brilliantly solved in \cite{BN} (for elliptic equations driven by the Laplacian operator) by {\em changing the topology} from $H^1(\Omega)$ to $C^1(\overline\Omega)$, and proving that local minimizers of the energy functional in both topologies do coincide. This useful result was extended to the $p$-Laplacian in \cite{GAPM} and, later, in \cite{BIU}, keeping the space $C^1(\overline\Omega)$ and exploiting the global $C^{1,\alpha}(\overline\Omega)$-regularity of weak solutions of $p$-Laplacian equations obtained in \cite{L1}.
\vskip2pt
\noindent
Passing to the fractional framework, the situation becomes more involved. Indeed, weak solutions of fractional order Dirichlet problems, despite being fairly regular in the interior of the domain, only attain $C^s(\overline\Omega)$ as an {\em optimal} global regularity, even for $p=2$. The space $C^1(\overline{\Omega})$ must then be replaced and a natural choice is the larger space $C^0_s(\overline{\Omega})$ defined in \eqref{cos} below, which is essentially the set of functions $u$ such that $u/\ds\in C^0(\overline{\Omega})$.  The distance function  is used to replace differentiation at the boundary by its weaker $s$-fractional counterpart defined in \eqref{hopf} and the key point in this framework is that weak solutions $u$ are such that $u/\ds$ admits a H\"older continuous extension to $\overline\Omega$ with natural estimates (see Proposition \ref{reg}). Such regularity result was recently proved for the whole range of exponents $p>1$ in \cite{IM} (extending previous partial results), and grants compactness of solutions in $\cs$, as soon as uniform $L^\infty(\Omega)$ bounds are available.
\vskip2pt
\noindent
The Sobolev vs.\ H\"older minimizers coincidence result is then rephrased for the fractional framework, by comparing the topologies of $\w$ and of  $\cs$. As a matter of fact, under very general assumptions, local minimizers of the energy functional in $\w$ and $\cs$, respectively, do coincide. Such coincidence was proved in \cite{IMS} for the degenerate fractional $p$-Laplacian ($p\ge 2$) by using a variant of the method of \cite{GAPM}. In this paper we come back to this result in Theorem \ref{svh} below, exploiting the recent regularity theorem of \cite{IM}, and extend the coincidence result to all $p>1$. This is done via a simplified approach inspired by \cite{BIU}, which is especially suitable for dealing with the singular case ($1<p<2$). We remark that Theorem \ref{svh} has a general use that goes beyond problem $(P_\lambda)$, for instance it is applied to a family of superlinear equations in \cite{ISV}.
\vskip2pt
\noindent
The structure of the paper is the following: in Section \ref{sec2} we recall some preliminary results on the fractional $p$-Laplacian; in Section \ref{sec3} we prove equivalence between Sobolev and H\"older minima of the energy functional; finally, Section \ref{sec4} is devoted to the study of problem $(P_\lambda)$ and the proof of Theorem \ref{pqr}.
\vskip4pt
\noindent
{\bf Notations:} For any $a\in \R$ and $q>0$ we set $a^q=|a|^{q-2}a$ if $a\neq 0$, $a^q=0$ if $a=0$. Throughout the paper, for any $U\subset\R^N$ we shall set $U^c=\R^N\setminus U$ and denote by $|U|$ the $N$-dimensional Lebesgue measure of $U$. For any two measurable functions $u,v:U\to\R$, $u\le v$ in $U$ will mean that $u(x)\le v(x)$ for a.e.\ $x\in U$ (and similar expressions). The positive (resp., negative) part of $u$ is denoted $u^+$ (resp., $u^-$). If $X$ is an ordered Banach space, then $X_+$ will denote its non-negative order cone and $B_r(x)$, $\overline{B}_r(x)$, $\partial B_r(x)$ will be the open ball, the closed ball, and the sphere, respectively, centered at $x$ with radius $r>0$. For all $r\in[1,\infty]$, $\|\cdot\|_r$ denotes the standard norm of $L^r(\Omega)$ (or $L^r(\R^N)$, which will be clear from the context). Every function $u$ defined in $\Omega$ will be identified with its $0$-extension to $\R^N$. Moreover, $C$ will denote a positive constant (whose value may change case by case).

\section{Preliminaries}\label{sec2}

\noindent
Here we recall some preliminary results on fractional Sobolev spaces, the fractional $p$-Laplacian, and related problems. First, for all open $U\subseteq\R^N$ and all measurable functions $u:U\to\R$, we define the Gagliardo seminorm
\[[u]_{s,p,U} = \Big[\iint_{U\times U}\frac{|u(x)-u(y)|^p}{|x-y|^{N+ps}}\,dx\,dy\Big]^\frac{1}{p}.\]
Accordingly we define the fractional Sobolev space
\[W^{s,p}(U) = \big\{u\in L^p(U):\,[u]_{s,p,U}<\infty\big\}.\]
From now on, let $\Omega\subset\R^N$ be a bounded domain with $C^{1,1}$-smooth boundary, and assume $ps<N$. In order to incorporate the Dirichlet condition, we define the subspace
\[\w = \big\{u\in W^{s,p}(\R^N):\,u=0 \ \text{in $\Omega$}\big\},\]
which is a uniformly convex, separable Banach space with norm $\|u\|_{\w}=[u]_{s,p,\R^N}$, whose dual space is denoted by $W^{-s,p'}(\Omega)$. The embedding $\w\hookrightarrow L^q(\Omega)$ is continuous for all $q\in[1,p^*_s]$ and compact for all $q\in[1,\p)$, where $\p=Np/(N-ps)$. We also introduce the following special space:
\[\widetilde{W}^{s,p}(\Omega) = \Big\{u\in L^p_{\rm loc}(\R^N):\,\exists\,U\Supset\Omega \ \text{s.t. $u\in W^{s,p}(U)$,}\,\int_{\R^N}\frac{|u(x)|^{p-1}}{(1+|x|)^{N+ps}}\, dx<\infty\Big\}.\]
For a comprehensive introduction to the fractional Sobolev spaces, see \cite{L}. We can define the fractional $p$-Laplacian as a nonlinear operator $\fpl:\widetilde{W}^{s,p}(\Omega)\to W^{-s,p'}(\Omega)$ by setting for all $u\in\widetilde{W}^{s,p}(\Omega)$, $\varphi\in\w$
\[\langle\fpl u,\varphi\rangle = \iint_{\R^N\times\R^N} \frac{(u(x)-u(y))^{p-1}(\varphi(x)-\varphi(y))}{|x-y|^{N+ps}}\,dx\,dy\]
(with the notation $a^{p-1}=|a|^{p-2}a$ established above). Such definition is equivalent to the one given in Section \ref{sec1}. Clearly we have $\w\subset\widetilde{W}^{s,p}(\Omega)$, and the restriction of $\fpl$ to $\w$ is a continuous, maximal monotone, $(S)_+$-type operator (see \cite[Lemma 2.1]{FI}). In the next lemma we collect some useful monotonicity properties of $\fpl$:

\begin{proposition}\label{mon}
The operator $\fpl$ has the following properties:
\begin{enumroman}
\item\label{mon1} for all $u\in L^p_{\rm loc}(\R^N)$ s.t.\ $[u]_{s,p,\R^N}<\infty$
\[[u]_{s,p,\R^N}^p \le \langle\fpl u,\pm u^\pm\rangle;\]
\item\label{mon2} for all $u,v\in\widetilde{W}^{s,p}(\Omega)$ s.t.\ $(u-v)^+\in\w$ and
\[\langle\fpl u-\fpl v,(u-v)^+\rangle \le 0,\]
we have $u\le v$ in $\Omega$;
\item\label{mon3} for all $u,v\in\w\cap L^\infty(\Omega)$, $q\ge 1$
\[\langle\fpl u-\fpl v,(u-v)^q\rangle \ge 0.\]
\end{enumroman}
\end{proposition}
\begin{proof}
For \ref{mon1} see the proof of \cite[Lemma 2.1]{IL} (the original result is stated for $u\in\w$, but the proofs works even in a more general case). For \ref{mon2} see \cite[Lemma 9]{LL}. For \ref{mon3}, it is enough to note that for all $a,b,c,d\in\R$ s.t.\ $a-b=c-d$ we have
\[(a^{p-1}-b^{p-1})(c^q-d^q) \ge 0.\]
Taking $a=u(x)-u(y)$, $b=v(x)-v(y)$, $c=u(x)-v(x)$, and $d=u(y)-v(y)$, we get
\[\iint_{\R^N\times\R^N}\frac{\big((u(x)-u(y)^{p-1}-(v(x)-v(y))^{p-1}\big)\big((u(x)-v(x))^q-(u(y)-v(y))^q\big)}{|x-y|^{N+ps}}\,dx\,dy \ge 0.\]
This concludes the proof.
\end{proof}

\noindent
Now we introduce a general class of Dirichlet problems driven by the fractional $p$-Laplacian:
\beq\label{dir}
\begin{cases}
\fpl u = f(x,u) & \text{in $\Omega$} \\
u = 0 & \text{in $\Omega^c$.}
\end{cases}
\eeq
Here $\Omega$, $p$, $s$ are as above, while the (generally non-autonomous) reaction $f$ is subject to the following structural hypothesis for some $q\in[1,\p]$:
\begin{itemize}[leftmargin=.6cm]
\item[${\bf H}_q$] $f:\Omega\times\R\to\R$ is a Carath\'eodory mapping, and there exists $C_0>0$ s.t.\ for a.e.\ $x\in\Omega$ and all $t\in\R$
\[|f(x,t)| \le C_0(1+|t|^{q-1}).\]
\end{itemize}
Note that ${\bf H}_q$ implies ${\bf H}_{q'}$ for all $1\le q<q'\le\p$, and that ${\bf H}_q$ includes the delicate case of {\em critical growth} when $q=\p$. We define (super, sub-) solutions of our problem as follows:

\begin{definition}\label{sol}
Let $f$ satisfy ${\bf H}_q$ for some $q\in[1,\p]$:
\begin{enumroman}
\item\label{sol1} $u\in\widetilde{W}^{s,p}(\Omega)$ is a (weak) supersolution of \eqref{dir} if $u\ge 0$ in $\Omega^c$ and for all $\varphi\in\w_+$
\[\langle\fpl u,\varphi\rangle \ge \int_\Omega f(x,u)\varphi\,dx;\]
\item\label{sol2} $u\in\widetilde{W}^{s,p}(\Omega)$ is a (weak) subsolution of \eqref{dir} if $u\le 0$ in $\Omega^c$ and for all $\varphi\in\w_+$
\[\langle\fpl u,\varphi\rangle \le \int_\Omega f(x,u)\varphi\,dx;\]
\item\label{sol3} $u\in\w$ is a (weak) solution of \eqref{dir} if it is both a super- and a subsolution, i.e., for all $\varphi\in\w$
\[\langle\fpl u,\varphi\rangle = \int_\Omega f(x,u)\varphi\,dx.\]
\end{enumroman}
\end{definition}

\noindent
Exactly as in the local case ($s=1$), solutions of \eqref{dir} satisfy almost-uniform {\em a priori} bounds if the exponent $q$ is critical, uniform bounds if $q$ is subcritical, and a special universal bound if $q$ is below $p$. Since a complete proof of such bounds seems to lack in the literature, we include it for the reader's convenience:

\begin{proposition}\label{apb}
Let ${\bf H}_q$ hold for some $q\in[1,\p]$, $u\in\w$ be a solution of \eqref{dir}. Then, $u\in L^\infty(\Omega)$. Moreover:
\begin{enumroman}
\item\label{apb1} if $q=\p$, then there exists $\eps_0>0$, depending on $\Omega,p,s,C_0$ s.t.\ for all $K>0$ satisfying
\[\int_{\{|u|\ge K\}}|u|^{\p}\,dx \le \eps_0,\]
we have $\|u\|_\infty\le C$, for a constant $C>0$ depending on $\Omega,p,s,C_0,\|u\|_{\p}$, and $K$;
\item\label{apb2} if $1\le q<\p$, then $\|u\|_\infty\le C$, for a constant $C>0$ depending on $\Omega,p,q,s,C_0$, and $\|u\|_{\p}$;
\item\label{apb3} if $1\le q<p$, then $\|u\|_{\w}\le C$, $\|u\|_\infty\le C$, for a constant $C>0$ depending on $\Omega,p,q,s,C_0$.
\end{enumroman}
\end{proposition}
\begin{proof}
 Clearly ${\bf H}_q$ implies ${\bf H}_{\p}$.  Test \eqref{dir} with $u$ and use  ${\bf H}_{\p}$ to get
 \[
 \|u\|_{\w}^p=\big\langle\fpl u,u\big\rangle =\int_\Omega f(x, u) u\, dx\le C\int_\Omega |u|+|u|^{\p}\, dx,
 \]
 so that by the embedding $\w\hookrightarrow L^{\p}(\Omega)$  and H\"older's inequality, it holds
 \[
 \frac{1}{C} \|u\|_{\p}^{p/\p}\le \|u\|_{\w}^p\le C\Big( \|u\|_{\p}+\|u\|_{\p}^{\p}\Big).
 \]
 Hence, a bound on $\|u\|_{\p}$ is equivalent to one on $\|u\|_{\w}$, and statement \ref{apb1} follows combining \cite[Theorem 3.3, Remark 3.4]{CMS}.
\vskip2pt
\noindent
To prove \ref{apb2}, let $\eps_0>0$ be as in \ref{apb1}. Also fix $K>2$, to be determined later. By Chebyshev's inequality  we have for any $u\in  L^{\p}(\Omega)$, $t> 0$
\beq\label{apb4}
\big|\big\{|u|\ge t\big\}\big| \le \frac{\|u\|_{\p}^{\p}}{t^{\p}}.
\eeq
Also, we have
\beq\label{apb5}
\int_{\{|u|\ge K\}}|u|^{\p}\,dx \le 4^{\p}\Big\|\Big(|u|-\frac{K}{2}\Big)^+\Big\|_{\p}^{\p}.
\eeq
Indeed, we may split the integral on the left hand side as follows:
\[\int_{\{|u|\ge K\}}|u|^{\p}\,dx = \sum_{i=0}^\infty\int_{\{2^iK\le |u|<2^{i+1}K\}}|u|^{\p}\,dx.\]
Further, we note that whenever $2^iK\le |u|<2^{i+1}K$ ($i\ge 0$) we have
\[|u|-\frac{K}{2} \ge \Big(2^i-\frac{1}{2}\Big)K \ge \frac{2^{i+1}}{4}K \ge \frac{|u|}{4},\]
hence
\begin{align*}
\int_{\{|u|\ge K\}}|u|^{\p}\,dx &\le 4^{\p}\sum_{i=0}^\infty\int_{\{2^iK\le |u|<2^{i+1}K\}}\Big[\Big(|u|-\frac{K}{2}\Big)^+\Big]^{\p}\,dx \\
&\le 4^{\p}\Big\|\Big(|u|-\frac{K}{2}\Big)^+\Big\|_{\p}^{\p}.
\end{align*}
Now let $u\in\w$ be a solution of \eqref{dir}. We begin from Proposition \ref{mon} \ref{mon1}, then we test \eqref{dir} with $(u-K/2)^+\in\w$, we apply ${\bf H}_q$ recalling that $K>2$:
\beq
\label{jd}
\begin{split}
\Big\|\Big(u-\frac{K}{2}\Big)^+\Big\|_{\w}^p &\le \Big\langle\fpl u,\Big(u-\frac{K}{2}\Big)^+\Big\rangle \\
&= \int_\Omega f(x,u)\Big(u-\frac{K}{2}\Big)^+\,dx \\
&\le C\int_{\{u>K/2\}}u^{q-1}\Big(u-\frac{K}{2}\Big)\,dx.
\end{split}
\eeq
On $\{u\ge K/2\}$ it holds
\[
u\le 2\Big(u-\frac{K}{4}\Big) \quad\text{and}\quad  u-\frac{K}{2}\le u-\frac{K}{4}
\]
so 
\[
u^{q-1}\Big(u-\frac{K}{2}\Big)\le 2^{q-1}\Big(u-\frac{K}{4}\Big)^q.
\]
Using the latter in \eqref{jd}, we can proceed through H\"older's inequality, to get
\begin{align*}
\Big\|\Big(u-\frac{K}{2}\Big)^+\Big\|_{\w}^p &\le C\int_{\{u>K/2\}}\Big(u-\frac{K}{4}\Big)^q\,dx\\
&\le C\Big\|\Big(u-\frac{K}{4}\Big)^+\Big\|_{\p}^q\Big|\Big\{u\ge\frac{K}{4}\Big\}\Big|^{1-\frac{q}{\p}}\\
&\le C\|u\|_{\p}^q\Big|\Big\{u\ge\frac{K}{4}\Big\}\Big|^{1-\frac{q}{\p}}.
\end{align*}
Similarly, testing with $-(u+K/2)^-\in\w$, we get
\[\Big\|\Big(u+\frac{K}{2}\Big)^-\Big\|_{\w}^p \le C\|u\|_{\p}^q\Big|\Big\{u\le-\frac{K}{4}\Big\}\Big|^{1-\frac{q}{\p}}.\]
Summing up these two estimates and using   the embedding $\w\hookrightarrow L^{\p}(\Omega)$, we infer
\[
\Big\|\Big(|u|-\frac{K}{2}\Big)^+\Big\|_{\p}^p \le C\|u\|_{\p}^q\Big|\Big\{|u|\ge\frac{K}{4}\Big\}\Big|^{1-\frac{q}{\p}} 
\]
and applying \eqref{apb4} we conclude
\[
\Big\|\Big(|u|-\frac{K}{2}\Big)^+\Big\|_{\p}^p \le \frac{C}{K^{\p-q}}\|u\|_{\p}^{\p}.
\]
Finally, we use \eqref{apb5}:
\begin{align*}
\int_{\{|u|\ge K\}}|u|^{\p}\,dx &\le C\Big\|\Big(|u|-\frac{K}{2}\Big)^+\Big\|_{\p}^{\p} \\
&\le \frac{C}{K^{(\p-q)\frac{\p}{p}}}\|u\|_{\p}^{(\p)^2/p},
\end{align*}
and the latter tends to $0$ as $K\to\infty$, since $q<\p$. So, we can choose $K>0$ depending on the data and $\|u\|_{\p}$, s.t.\
\[\int_{\{|u|\ge K\}}|u|^{\p}\,dx \le \eps_0.\]
By \ref{apb1}, then we deduce $\|u\|_\infty\le C$, with $C>0$ depending on $\Omega,p,q,s,C_0,$ and $\|u\|_{\p}$.
\vskip2pt
\noindent
Finally, we prove \ref{apb3}. If ${\bf H}_q$ holds with $q<p$, then for any solution $u\in\w$ of \eqref{dir} we have
\begin{align*}
\|u\|_{\w}^p &= \langle\fpl u,u\rangle \\
&= \int_\Omega f(x,u)u\,dx \\
&\le C_0\int_\Omega(|u|+|u|^q)\,dx \\
&\le C\|u\|_1+C\|u\|_q^q.
\end{align*}
By the continuous embeddings $\w\hookrightarrow L^1(\Omega),L^q(\Omega)$, and $p>q$, we see that $\|u\|_{\w}\le C$ for some $C>0$ depending on $\Omega,p,q,s,C_0$. In turn, by \ref{apb2}, this implies $\|u\|_\infty\le C$ for a different $C>0$ depending on the same quantities.
\end{proof}

\noindent
Classical nonlinear regularity theory does not apply to fractional equations. Combining Proposition \ref{apb} above with \cite[Theorem 2.7]{IM}, one can see that any solution $u\in\w$ of \eqref{dir} enjoys the optimal global regularity $u\in C^s(\R^N)$. More important, to our purposes, is a form of fine (or weighted) boundary regularity, already mentioned in Section \ref{sec1}. We now discuss such issue in more precise terms. For future use, we prefer to deal with a simplified problem:
\beq\label{dirg}
\begin{cases}
\fpl u = g(x) & \text{in $\Omega$} \\
u = 0 & \text{in $\Omega^c$,}
\end{cases}
\eeq
whose solutions are meant as in Definition \ref{sol}. Set for all $x\in\R^N$
\[{\rm d}(x) = {\rm dist}(x,\Omega^c).\]
The following fine boundary regularity result holds (see \cite[Theorem 1.1]{IM}):

\begin{proposition}\label{reg}
Let $g\in L^\infty(\Omega)$, $u\in\w$ be a solution of \eqref{dirg}. Then, $u/\ds$ admits a $\alpha$-H\"older continuous extension to $\overline\Omega$, satisfying
\[\Big\|\frac{u}{\ds}\Big\|_{C^\alpha(\overline\Omega)} \le C\|g\|_\infty^\frac{1}{p-1},\]
with $\alpha\in(0,s)$, $C>0$ depending on $N,p,s,$ and $\Omega$.
\end{proposition}

\noindent
We rephrase the discussion above in terms of function spaces. Set for all $\alpha\in[0,1)$
\beq
\label{cos}
C^\alpha_s(\overline\Omega) = \Big\{u\in C(\overline\Omega):\,\frac{u}{\ds} \ \text{has an extension in $C^\alpha(\overline\Omega)$}\Big\},
\eeq
a Banach space endowed with the norm
\[\|u\|_{C^\alpha_s(\overline\Omega)} = \Big\|\frac{u}{\ds}\Big\|_{C^\alpha(\overline\Omega)}.\]
First, note that the positive order cone $\cs_+$ has a nonempty interior given by
\beq
\label{intcs}
{\rm int}(\cs_+) = \Big\{u\in\cs:\,\inf_{x\in\Omega}\frac{u(x)}{\ds(x)}>0\Big\}.
\eeq
By the Ascoli-Arzel\`a theorem, the embedding $C^\alpha_s(\overline\Omega)\hookrightarrow C^\beta_s(\overline\Omega)$ is compact for all $\alpha>\beta$. Proposition \ref{reg} ensures that any solution of \eqref{dirg} lies in $C^\alpha_s(\overline\Omega)$, with a control on the norm depending on the $L^\infty(\Omega)$-norm of the reaction $g$.
\vskip2pt
\noindent
Now let $u\in\w$ be a solution of \eqref{dir}. By applying Proposition \ref{apb} with a convenient $K>0$ we find $u\in L^\infty(\Omega)$. This in turn, through ${\bf H}_q$, implies $f(\cdot,u)\in L^\infty(\Omega)$. Proposition \ref{reg} (with $g=f(\cdot,u)$) then ensures that $u\in C^\alpha_s(\overline\Omega)$. Note that a uniform bound of the $C^\alpha_s(\overline\Omega)$-norm of $u$ is only available if $q<\p$.
\vskip2pt
\noindent
We conclude this section by recalling the strong minimum and comparison principles for the fractional $p$-Laplacian. From \cite[Theorems 2.6, 2.7]{IMP} we have:

\begin{proposition}\label{smp}
Let $h\in C^0(\R)\cap BV_{\rm loc}(\R)$, $u\in\widetilde{W}^{s,p}(\Omega)\cap C^0(\overline\Omega)$, $u\neq 0$, s.t.\
\[\begin{cases}
\fpl u+h(u) \ge h(0) & \text{in $\Omega$} \\
u \ge 0 & \text{in $\R^N$.}
\end{cases}\]
Then,
\[\inf_\Omega\,\frac{u}{\ds} > 0.\]
In particular, if $u\in\cs$, then $u\in{\rm int}(\cs_+)$.
\end{proposition}

\begin{proposition}\label{scp}
Let $h\in C^0(\R)\cap BV_{\rm loc}(\R)$, $u\in\widetilde{W}^{s,p}(\Omega)\cap C^0(\overline\Omega)$, $v\in\w\cap C^0(\overline\Omega)$, $u\neq v$, $K>0$ s.t.\
\[\begin{cases}
\fpl v+h(v) \le \fpl u+h(u) \le K & \text{in $\Omega$} \\
0 < v \le u & \text{in $\Omega$} \\
u \ge 0 & \text{in $\Omega^c$.}
\end{cases}\]
Then, $u>v$ in $\Omega$. In particular, if $u\in{\rm int}(\cs_+)$, then $u-v\in{\rm int}(\cs_+)$.
\end{proposition}

\section{Sobolev vs.\ H\"older minima}\label{sec3}

\noindent
Solutions of problem \eqref{dir} can be interpreted as critical points of an energy functional. In this section we assume that $f$ satisfies ${\bf H}_{\p}$. Set for all $(x,t)\in\Omega\times\R$
\[F(x,t) = \int_0^t f(x,\tau)\,d\tau.\]
Due to ${\bf H}_{\p}$, we may set for all $u\in\w$
\[\Phi(u) = \frac{\|u\|_{\w}^p}{p}-\int_\Omega F(x,u)\,dx,\]
so that $\Phi\in C^1(\w)$ with Fr\'echet derivative given for all $u,\varphi\in\w$ by
\[\langle\Phi'(u),\varphi\rangle = \langle\fpl u,\varphi\rangle-\int_\Omega f(x,u)\varphi\,dx.\]
Thus, Definition \ref{sol} \ref{sol3} is clearly equivalent to $\Phi'(u)=0$ in $W^{-s,p'}(\Omega)$. In the study of critical points of $\Phi$, we often begin with a local minimizer. In this connection, it is sometimes useful to switch from the topology of $\w$ to that of $\cs$, in order to deal with truncations.
\vskip2pt
\noindent
We can now prove the equivalence of Sobolev and H\"older local minima of the functional $\Phi$:

\begin{theorem}\label{svh}
Let $f$ satisfy ${\bf H}_{\p}$, $u_0\in\w$. Then, the following conditions are equivalent:
\begin{enumroman}
\item\label{svh1} there exists $\rho>0$ s.t.\ $\Phi(u_0+v)\ge\Phi(u_0)$ for all $v\in\w$, $\|v\|_{\w}\le\rho$;
\item\label{svh2} there exists $\sigma>0$ s.t.\ $\Phi(u_0+v)\ge\Phi(u_0)$ for all $v\in\w\cap\cs$, $\|v\|_{\cs}\le\sigma$.
\end{enumroman}
\end{theorem}
\begin{proof}
First we prove that \ref{svh1} implies \ref{svh2}. From \ref{svh1} and $\Phi\in C^1(\w)$ we deduce that for all $\varphi\in\w$
\[\langle\Phi'(u_0),\varphi\rangle = \lim_{\tau\to 0^+}\frac{\Phi(u_0+\tau\varphi)-\Phi(u_0)}{\tau} \ge 0.\]
Replacing $\varphi$ with $-\varphi$, we find $\Phi'(u_0)=0$. So $u_0$ is a solution of \eqref{dir}. By Proposition \ref{apb} we have $u_0\in L^\infty(\Omega)$, hence by Proposition \ref{reg} we deduce $u_0\in C^\alpha_s(\overline\Omega)$.
\vskip2pt
\noindent
We now argue by contradiction. Assume that there exists a sequence $(u_n)$ in $\w\cap\cs$ s.t.\ $u_n\to u$ in $\cs$, and for all $n\in\N$
\beq\label{svh3}
\Phi(u_n) < \Phi(u_0).
\eeq
In particular $u_n\to u_0$ in $L^\infty(\Omega)$, so by ${\bf H}_{\p}$ and dominated convergence we have
\[\lim_n\,\int_\Omega F(x,u_n)\,dx = \int_\Omega F(x,u_0)\,dx.\]
In turn, by \eqref{svh3} we get
\begin{align*}
\limsup_n\frac{\|u_n\|_{\w}^p}{p} &\le \limsup_n\Phi(u_n)+\lim_n\,\int_\Omega F(x,u_n)\,dx \\
&\le \Phi(u_0)+\int_\Omega F(x,u_0)\,dx \le \frac{\|u_0\|_{\w}^p}{p}.
\end{align*}
In particular, $(u_n)$ bounded in $\w$. By reflexivity, up to a relabeled subsequence we have $u_n\rightharpoonup u_0$ (recall that $u_n\to u_0$ in $L^\infty(\Omega)$). Therefore,
\begin{align*}
\|u_0\|_{\w} &\le \liminf_n\|u_n\|_{\w} \\
&\le \limsup_n\|u_n\|_{\w} \le \|u_0\|_{\w},
\end{align*}
i.e., $\|u_n\|_{\w}\to\|u_0\|_{\w}$. By uniform convexity we have $u_n\to u_0$ in $\w$. Thus, for all $n\in\N$ big enough $\|u_n-u_0\|_{\w}\le\rho$, which by \ref{svh1} implies
\[\Phi(u_n) \ge \Phi(u_0),\]
against \eqref{svh3}. This proves \ref{svh2}.
\vskip2pt
\noindent
Now we prove that \ref{svh2} implies \ref{svh1} (which is more involved). By \ref{svh2}, as above for all $\varphi\in\w\cap\cs$ we have
\[\langle\Phi'(u_0),\varphi\rangle \ge 0,\]
hence $\Phi'(u_0)\in W^{-s,p'}(\Omega)$ vanishes along all directions in $\w\cap\cs$. Since the latter is a dense subspace of $\w$, we have $\Phi'(u_0)=0$. As in the previous case, using Propositions \ref{apb} and \ref{reg} we find $u_0\in C^\alpha_s(\overline\Omega)$.
\vskip2pt
\noindent
Again we argue by contradiction, assuming that there exists a sequence $(\tilde u_n)$ in $\w$ s.t.\ $\tilde u_n\to u_0$ in $\w$, and for all $n\in\N$
\beq\label{svh4}
\Phi(\tilde u_n) < \Phi(u_0).
\eeq
For all $n\in\N$ set $\delta_n=\|\tilde u_n-u_0\|_{\p}>0$ and
\[B_n = \big\{u\in\w:\,\|u-u_0\|_{\p}\le\delta_n\big\}.\]
By the continuous embedding $\w\hookrightarrow L^{\p}(\Omega)$ we have $\delta_n\to 0$, and $B_n$ is a closed convex (hence, weakly closed) subset of $\w$. Anyway, $B_n$ is not weakly compact, nor $\Phi$ is sequentially weakly l.s.c.\ due to the possibly critical growth in ${\bf H}_{\p}$, so we must introduce a truncation.
\vskip2pt
\noindent
For all $\kappa>0$, $t\in\R$ set
\[[t]_\kappa = \begin{cases}
-\kappa & \text{if $t\le -\kappa$} \\
t & \text{if $-\kappa<t<\kappa$} \\
\kappa & \text{if $t\ge\kappa$.}
\end{cases}\]
By dominated convergence, for all $u\in\w$ we have
\[\lim_{\kappa\to\infty}\,\int_\Omega\int_0^u f(x,[t]_\kappa)\,dt\,dx = \int_\Omega F(x,u)\,dx.\]
By \eqref{svh4}, for all $n\in\N$ we may find $\eps_n\in(0,\Phi(u_0)-\Phi(\tilde u_n))$ and $\kappa_n>\|u_0\|_\infty+1$ s.t.\
\[\Big|\int_\Omega\int_0^{\tilde u_n} f(x,[t]_{\kappa_n})\,dt\,dx-\int_\Omega F(x,\tilde u_n)\,dx\Big| < \eps_n.\]
Set for all $(x,t)\in\Omega\times\R$
\[f_n(x,t) = f(x,[t]_{\kappa_n}), \qquad F_n(x,t) = \int_0^t f_n(x,\tau)\,d\tau.\]
By ${\bf H}_{\p}$ we have
\[|f_n(x,t)| \le C_0(1+\kappa_n^{\p-1}),\]
so we may define a truncated energy functional by setting for all $u\in\w$
\[\Phi_n(u) = \frac{\|u\|_{\w}^p}{p}-\int_\Omega F_n(x,u)\,dx.\]
By the estimate above on $f_n$, $\Phi_n\in C^1(\w)$ is sequentially weakly l.s.c.\ and coercive. So, there exists $u_n\in B_n$ s.t.\
\[\Phi_n(u_n) = \inf_{u\in B_n}\Phi_n(u).\]
By the choice of $\kappa_n$ we have $\Phi_n(u_0)=\Phi(u_0)$, and by construction
\[|\Phi_n(\tilde u_n)-\Phi(\tilde u_n)| = \Big|\int_\Omega F_n(x,\tilde u_n)\,dx-\int_\Omega F(x,\tilde u_n)\,dx\Big| < \eps_n.\]
Concatenating the inequalities and recalling the choice of $\eps_n$, we get
\[\Phi_n(u_n) \le \Phi_n(\tilde u_n) \le \Phi(\tilde u_n)+\eps_n < \Phi(u_0) = \Phi_n(u_0),\]
so that \eqref{svh4} is in a sense transferred onto the truncated functional $\Phi_n$. We will now write a Euler-Lagrange equation for the minimization problem solved by $u_n$. We claim that for all $n\in\N$ there exists $\lambda_n\ge 0$ s.t.\ $u_n$ is a weak solution of
\beq\label{svh5}
\begin{cases}
\fpl u_n+\lambda_n(u_n-u_0)^{\p-1} = f_n(x,u_n) & \text{in $\Omega$} \\
u_n = 0 & \text{in $\Omega^c$,}
\end{cases}
\eeq
in the sense of Definition \ref{sol} \ref{sol3}. Indeed, fix $n\in\N$ and distinguish two cases:
\begin{itemize}[leftmargin=1cm]
\item[$(a)$] If $\|u_n-u_0\|_{\p}<\delta_n$, then $u_n\in{\rm int}(B_n)$ (with respect to the $\w$-topology, by virtue of continuity of $\|\cdot\|_{\p}$) is a local minimizer of $\Phi_n$. Therefore we have
\[\Phi_n'(u_n) = 0,\]
i.e., \eqref{svh5} holds with $\lambda_n=0$.
\item[$(b)$] If $\|u_n-u_0\|_{\p}=\delta_n$, then $u_n$ minimizes $\Phi_n$ under the constraint
\[\Psi(u) := \frac{\|u-u_0\|_{\p}^{\p}}{\p} = \frac{\delta_n^{\p}}{\p}.\]
Note that $\Psi\in C^1(\w)$ with Fr\'echet derivative given for all $u,\varphi\in\w$ by
\[\langle\Psi'(u),\varphi\rangle = \int_\Omega(u-u_0)^{\p-1}\varphi\,dx.\]
By Lagrange's multipliers rule, there exists $\lambda_n\in\R$ s.t.\
\[\Phi'_n(u_n)+\lambda_n\Psi'(u_n) = 0,\]
i.e., \eqref{svh5} holds. It remains to determine the sign of $\lambda_n$. By minimization we have
\[\langle\Phi'_n(u_n),u_0-u_n\rangle = \lim_{\tau\to 0^+}\,\frac{\Phi_n(u_n+\tau(u_0-u_n))-\Phi_n(u_n)}{\tau} \ge 0.\]
Besides, since $u_n\in\partial B_n$ maximizes $\Psi$ over $B_n$, we have
\[\langle\Psi'(u_n),u_0-u_n\rangle = \lim_{\tau\to 0^+}\frac{\Psi(u_n+\tau(u_0-u_n))-\Psi(u_n)}{\tau} < 0.\]
Combining such inequalities we get
\[\lambda_n = -\frac{\langle\Phi'_n(u_n),u_0-u_n\rangle}{\langle\Psi'(u_n),u_0-u_n\rangle} \ge 0.\]
\end{itemize}
In either case $u_n$ solves the perturbed problem \eqref{svh5} for all $n\in\N$, with a (possibly unbounded) sequence $(\lambda_n)$ in $[0,+\infty[$. Regarding the solutions $(u_n)$, all we know for now is that $u_n\to u_0$ in $L^{\p}(\Omega)$.
\vskip2pt
\noindent
The rest of the proof is based on two uniform bounds on $(u_n)$. First we prove that there exists $C_1>0$, depending on $N,p,s,\Omega$, and $u_0$, s.t.\ for all $n\in\N$ (up to a subsequence) we have
\beq\label{svh6}
\|u_n\|_\infty \le C_1.
\eeq
Indeed, by ${\bf H}_{\p}$ we have for all $n\in\N$, a.e.\ $x\in\Omega$, and all $t\in\R$
\[|f_n(x,t)| \le C_0(1+|t|^{\p-1}),\]
with $C_0>0$ independent of $n$. Set for all $n\in\N$, $(x,t)\in\Omega\times\R$
\[g_n(x,t) = f_n(x,t)-\lambda_n(t-u_0(x))^{\p-1}.\]
Again we distinguish two cases:
\begin{itemize}[leftmargin=1cm]
\item[$(a)$] If $\lambda_n>C_0+1$, then for a.e.\ $x\in\Omega$ and all $t>0$ we have
\begin{align*}
g_n(x,t) &\le C_0(1+t^{\p-1})-\lambda_n(t-\|u_0\|_\infty)^{\p-1} \\
&\le (C_0-\lambda_n)t^{\p-1}+C \le -t^{\p-1}+C,
\end{align*}
for $C=C(N, p, s, u_0)$, and so for all $t>C_1(N, p, s, u_0)$  it holds
\[g_n(x,t) < 0.\]
Similarly, for a bigger $C_1>0$ if necessary, we have for a.e.\ $x\in\Omega$ and all $t<-C_1$
\[g_n(x,t) > 0.\]
Note that $C_1>0$ does not depend on $n$. The constant $C_1$ lies in $\widetilde{W}^{s,p}(\Omega)$. So, we test \eqref{svh5} with $(u_n-C_1)_+\in\w$ and apply Proposition \ref{mon} \ref{mon2}:
\[\langle\fpl u-\fpl C_1,(u-C_1)^+\rangle = \int_{\{u_n>C_1\}}g_n(x,u_n)(u_n-C_1)\,dx \le 0.\]
Therefore, $u_n\le C_1$ in $\Omega$. Similarly we get $u_n\ge -C_1$, thus proving \eqref{svh6}.
\item[$(b)$] If $\lambda_n\in[0,C_0+1]$, then we invoke ${\bf H}_{\p}$ to find $C(N, p, s, u_0)>0$  s.t.\ for all such $n\in\N$ it holds
\[|g_n(x,t)| \le C(1+|t|^{\p-1})\]
for a.e.\ $x\in\Omega$, and all $t\in\R$.
Let $\eps_0>0$ (depending on such $C$, and hence on $N,p,s,\Omega,$ and $u_0$) be as in Proposition \ref{apb} \ref{apb1}, and fix $K=2\|u_0\|_\infty>0$. Recalling that $u_n\to u_0$ in $L^{\p}(\Omega)$, for all $n\in\N$ big enough we have
\[\int_{\{|u_n|\ge K\}}|u_n|^{\p}\,dx < \eps_0.\]
Besides, $u_n$ solves \eqref{svh5}. By Proposition \ref{apb} \ref{apb1}, we can find $C_1>0$ (depending ultimately on $N,p,s,\Omega$, and $u_0$) s.t.\ \eqref{svh6} holds for all $n\in\N$.
\end{itemize}
This concludes the proof of estimate \eqref{svh6}, which  unfortunately does not suffice to control the perturbation $\lambda_n(u_n-u_0)^{\p-1}$ in \eqref{svh5}, as $(\lambda_n)$ may be unbounded. Hence, we need a second uniform bound as follows. There exists $C_2>0$, depending on $N,p,s,\Omega$, and $u_0$, s.t.\ for all $n\in\N$
\beq\label{svh7}
\lambda_n\|u_n-u_0\|_\infty^{\p-1} \le C_2.
\eeq
Indeed, fix $q\ge 1$. Testing \eqref{dir} and \eqref{svh5} with $(u_n-u_0)^q$ (which belongs to $\w$ by the boundedness of $u_n-u_0$) and using Proposition \ref{mon} \ref{mon3}, we have for all $n\in\N$
\begin{align*}
\lambda_n\int_\Omega|u_n-u_0|^{\p+q-1}\,dx &\le \langle\fpl u_n-\fpl u_0,(u_n-u_0)^q\rangle+\lambda_n\int_\Omega|u_n-u_0|^{\p+q-1}\,dx \\
&= \int_\Omega\big[f_n(x,u_n)-f(x,u_0)\big](u_n-u_0)^q\,dx \\
&\le C(1+\|u_0\|_\infty^{\p-1}+\|u_n\|_\infty^{\p-1})\|u_n-u_0\|_q^q.
\end{align*}
We use now \eqref{svh6} and H\"older's inequality to find $C_2>0$ depending on $N,p,s,\Omega$, and $u_0$, s.t.\ for all $n\in\N$, $q\ge 1$
\begin{align*}
\lambda_n\int_\Omega|u_n-u_0|^{\p+q-1}\,dx &\le C_2\int_\Omega|u_n-u_0|^q\,dx \\
&\le C_2\Big[\int_\Omega|u_n-u_0|^{\p+q-1}\,dx\Big]^\frac{q}{\p+q-1}\,|\Omega|^{1-\frac{q}{\p+q-1}},
\end{align*}
which rephrases as
\[\lambda_n\|u_n-u_0\|_{\p+q-1}^{\p-1} \le C_2|\Omega|^\frac{\p-1}{\p+q-1}.\]
Letting $q\to\infty$, we achieve \eqref{svh7}.
\vskip2pt
\noindent
We can now complete our proof. By the bounds \eqref{svh6} and \eqref{svh7}, we have for all $n\in\N$, a.e.\ $x\in\Omega$
\[|g_n(x,u_n)| \le C(1+\|u_n\|_\infty^{\p-1})+\lambda_n\|u_n-u_0\|_\infty^{\p-1} \le C,\]
with $C>0$ depending on $N,p,s,\Omega,$ and $u_0$. Hence, by Proposition \ref{reg} (with $g=g_n(\cdot,u_n)$) we have $u_n\in C^\alpha_s(\overline\Omega)$ with the bound
\[\|u_n\|_{C^\alpha_s(\overline\Omega)} \le C,\]
with $\alpha\in(0,s)$, $C>0$ independent of $n$. By the compact embedding $C^\alpha_s(\overline\Omega)\hookrightarrow\cs$, $(u_n)$ has a convergent subsequence in $\cs$. Recalling that $u_n\to u_0$ in $L^{\p}(\Omega)$, we have in fact $u_n\to u_0$ in $\cs$. Thus, for all $n\in\N$ big enough we have
\[\|u_n-u_0\|_{\cs} \le \sigma,\]
with $\sigma>0$ as in \ref{svh2}. In particular, $u_n\to u_0$ in $L^\infty(\Omega)$, hence for all $n\in\N$ big enough
\[\|u_n\|_\infty \le \|u_0\|_\infty+1 < \kappa_n.\]
By construction, then, we have for a.e.\ $x\in\Omega$
\[f_n(x,u_n) = f(x,u_n),\]
which in turn implies $\Phi_n(u_n)=\Phi(u_n)$. By the last relations and \eqref{svh4}, for all $n\in\N$ big enough we have
\[\Phi(u_n) < \Phi(u_0),\]
against \ref{svh2}. This proves \ref{svh1}, and concludes the argument.
\end{proof}

\begin{remark}\label{brock}
The approach followed here to prove Theorem \ref{svh} differs from that used in \cite{IMS}, in the sense that no strong monotonicity property is employed. In \cite{IMS}, the assumption $p\ge 2$ allowed a strong monotonicity inequality (Lemma 2.3), but for the singular case such property does not seem to be easily achieved. The simpler proof presented here is inspired by \cite{BIU}, with the necessary adaptations to the fractional, possibly critical setting.
\end{remark}

\section{Bifurcation result}\label{sec4}

\noindent
In this section we finally come back to problem $(P_\lambda)$ and we prove Theorem \ref{pqr}. Set
\[\Lambda = \big\{\lambda>0:\,\text{$(P_\lambda)$ has a solution}\big\}.\]
Also, for all $\lambda\in\Lambda$ we set
\[\mathcal{S}(\lambda) = \big\{u\in\w:\,\text{$u$ solution of $(P_\lambda)$}\big\}.\]
Most of the section will be devoted to studying the properties of the sets defined above. Fix $\lambda>0$. Set for all $t\in\R$
\[f_\lambda(t) = \lambda(t^+)^{q-1}-(t^+)^{r-1}.\]
Clearly $f_\lambda:\R\to\R$ satisfies ${\bf H}_q$ (with $q<p$). Set further
\[F_\lambda(t) = \int_0^t f_\lambda(\tau)\,d\tau,\]
and for all $u\in\w$
\[\Phi_\lambda(u) = \frac{\|u\|_{\w}^p}{p}-\int_\Omega F_\lambda(u)\,dx.\]
As in Section \ref{sec3}, $\Phi_\lambda\in C^1(\w)$ and its critical points satisfy weakly in $\Omega$
\[\fpl u = f_\lambda(u).\]
For clarity, we will divide the proof in several steps. Our first result states the non-existence of (positive) solutions of problem $(P_\lambda)$, for $\lambda>0$ small:

\begin{lemma}\label{non}
For all $\lambda>0$ small enough, $(P_\lambda)$ has no solutions.
\end{lemma}
\begin{proof}
Let $\hat\lambda_1>0$ be the principal eigenvalue of $\fpl$ in $\w$ (see \cite{LL}) and choose $\eps\in(0,\hat\lambda_1)$. Setting  $\lambda_0=\min\{1, \eps\}>0$, we claim that for all $\lambda\in(0,\lambda_0)$ and all $t\ge 0$ it holds
\beq\label{non1}
f_\lambda(t) \le \eps t^{p-1}.
\eeq
Indeed, from $\lambda\le \lambda_0\le  1$ and $0<r<q$ we immediately  infer $f_\lambda(t)\le 0$ for  $t\in [0, 1]$. For $t\ge 1$ we instead have, due to $p>q$ and $\lambda\le \lambda_0\le \eps$, that
\[f_\lambda(t) \le \lambda t^{q-1} \le \eps t^{p-1},\]
concluding the proof of \eqref{non1}. Now let $\lambda\in(0,\lambda_0)$, $u\in\w$ be a critical point of $\Phi_\lambda$. Then, by \eqref{non1} we have
\begin{align*}
\|u\|_{\w}^p &= \langle\fpl u,u\rangle \\
&= \int_\Omega f_\lambda(u)u\,dx \le \eps\|u\|_p^p.
\end{align*}
By $\eps<\hat\lambda_1$ and the variational characterization of the principal eigenvalue, we have $u=0$. So, $(P_\lambda)$ has no (positive) solutions, or equivalently, $\lambda\notin\Lambda$.
\end{proof}

\begin{lemma}\label{exi}
For all $\lambda>0$ big enough, $(P_\lambda)$ has at least one solution.
\end{lemma}
\begin{proof}
Fix $\lambda>0$. The functional $\Phi_\lambda$, defined above, is sequentially weakly l.s.c.\ in $\w$. Also, $\Phi_\lambda$ is coercive in $\w$. Indeed, by the continuous embedding $\w\hookrightarrow L^q(\Omega)$, for all $u\in\w$ we have
\begin{align*}
\Phi_\lambda(u) &= \frac{\|u\|_{\w}^p}{p}-\lambda\frac{\|u\|_q^q}{q}+\frac{\|u\|_r^r}{r} \\
&\ge \frac{\|u\|_{\w}^p}{p}-C\lambda\|u\|_{\w}^q,
\end{align*}
and the latter tends to $\infty$ as $\|u\|_{\w}\to\infty$ (recall that $q<p$). So, there exists $u_\lambda\in\w$ s.t.\
\[\Phi_\lambda(u_\lambda) = \inf_{u\in\w}\Phi_\lambda(u) =: m_\lambda.\]
In particular, $u_\lambda$ is a critical point of $\Phi_\lambda$, i.e., a weak solution of
\beq\label{exi1}
\fpl u_\lambda = f_\lambda(u_\lambda).
\eeq
By Proposition \ref{reg} we have $u_\lambda\in C^\alpha_s(\overline\Omega)$. Testing \eqref{exi1} with $-u_\lambda^-\in\w$ and applying Proposition \ref{mon} \ref{mon1}, we get
\begin{align*}
\|u_\lambda^-\|_{\w}^p &\le \langle\fpl u_\lambda,-u_\lambda^-\rangle \\
&= \int_\Omega f_\lambda(u_\lambda)(-u_\lambda^-)\,dx \\
&= \int_{\{u_\lambda<0\}}f_\lambda(u_\lambda)u_\lambda\,dx = 0.
\end{align*}
Therefore, $u_\lambda\ge 0$ in $\Omega$. Now fix any $\bar u\in\w_+\setminus\{0\}$, then we have
\[\Phi_\lambda(\bar u) = \frac{\|\bar u\|_{\w}^p}{p}-\lambda\frac{\|\bar u\|_q^q}{q}+\frac{\|\bar u\|_r^r}{r},\]
and the latter tends to $-\infty$ as $\lambda\to\infty$. So, for all $\lambda>0$ big enough we have
\[m_\lambda \le \Phi_\lambda(\bar u) < 0,\]
hence $u_\lambda\neq 0$. By Proposition \ref{smp} (with $h=-f_\lambda$) we get $u_\lambda\in{\rm int}(\cs_+)$. So we have proved that $\lambda\in\Lambda$ for all $\lambda>0$ big enough.
\end{proof}

\noindent
By Lemmas \ref{non}, \ref{exi} we may set
\[\lambda_* = \inf\,\Lambda > 0.\]
We will now prove that $\lambda_*\in\Lambda$:

\begin{lemma}\label{str}
There exists a solution $u_*\in{\rm int}(\cs_+)$ of $(P_{\lambda_*})$.
\end{lemma}
\begin{proof}
Let $(\lambda_n)$ be a sequence in $\Lambda$, s.t.\ $\lambda_n\to\lambda_*$. Clearly, for all $n\in\N$ we have $\lambda_n\le M$ ($M>0$ independent of $n$), and there exists a solution $u_n\in{\rm int}(\cs_+)$ of $(P_{\lambda_n})$. For all $n\in\N$, $t\in\R$ we have
\begin{align*}
|f_{\lambda_n}(t)| &\le M|t|^{q-1}+|t|^{r-1} \\
&\le C(1+|t|^{q-1}),
\end{align*}
with $C>0$ independent of $n$ and $q<p$. By Proposition \ref{apb} \ref{apb3}, we have $\|u_n\|_{\w},\|u_n\|_\infty \le C$. Due to reflexivity of $\w$ and the compact embeddings $\w\hookrightarrow L^q(\Omega),\,L^r(\Omega)$, passing to a subsequence we have $u_n\rightharpoonup u_*$ in $\w$, $u_n\to u_*$ in both $L^q(\Omega)$ and $L^r(\Omega)$, and $u_n(x)\to u_*(x)$ for a.e.\ $x\in\Omega$, hence $u_*\ge 0$ in $\Omega$.
\vskip2pt
\noindent
Test $(P_{\lambda_n})$ with $u_n-u_*$:
\begin{align*}
\langle\fpl u_n,u_n-u_*\rangle &= \int_\Omega(\lambda_nu_n^{q-1}-u_n^{r-1})(u_n-u_*)\,dx \\
&\le M\|u_n\|_q^{q-1}\|u_n-u_*\|_q+\|u_n\|_r^{r-1}\|u_n-u_*\|_r,
\end{align*}
and the latter tends to $0$ as $n\to\infty$. By the $(S)_+$-property of $\fpl$, we have $u_n\to u_*$ in $\w$. Besides, by Proposition \ref{reg} and the uniform $L^\infty(\Omega)$-bound on $(u_n)$, we deduce that $(u_n)$ is bounded in $C^\alpha_s(\overline\Omega)$. Passing to a further subsequence, we have $u_n\to u_*$ in $\cs$. We can now pass to the limit in $(P_{\lambda_n})$ and see that weakly in $\Omega$
\[\fpl u_* = f_{\lambda_*}(u_*).\]
To conclude, we need to prove that $u_*\neq 0$. Arguing by contradiction, assume that $u_n\to 0$ uniformly on $\overline\Omega$. Then, for all $n\in\N$ big enough we have
\[\|u_n\|_\infty \le \Big(\frac{1}{2M}\Big)^\frac{1}{q-r},\]
which in turn implies
\[\|u_n\|_q^q \le \|u_n\|_\infty^{q-r}\|u_n\|_r^r \le \frac{\|u_n\|_r^r}{2M}.\]
Testing $(P_{\lambda_n})$ with $u_n$ we get then
\begin{align*}
\|u_n\|_{\w}^p &= \langle\fpl u_n,u_n\rangle \\
&= \lambda_n\|u_n\|_q^q-\|u_n\|_r^r \\
&\le -\frac{\|u_n\|_r^r}{2} < 0,
\end{align*}
a contradiction. So $u_*\neq 0$. Finally, we have weakly in $\Omega$
\[\fpl u_*+u_*^{r-1} = \lambda_*u_*^{q-1} \ge 0,\]
so by Proposition \ref{smp} (with $h(t)=(t^+)^{r-1}$) we deduce $u_*\in{\rm int}(\cs_+)$. Thus, $\lambda_*\in\Lambda$ and $u_*\in\mathcal{S}(\lambda_*)$.
\end{proof}

\begin{remark}\label{comp}
Arguing as in Lemma \ref{str} above, one can easily retrieve more general information about the sets $\Lambda$ and $\mathcal{S}(\lambda)$. First, considering a constant sequence $\lambda_n=\lambda$, we see that $\mathcal{S}(\lambda)$ is compact in both $\w$ and $\cs$ for all $\lambda\in\Lambda$. Besides, taking a sequence $\lambda_n\to\lambda$, we see that $\lambda\in\Lambda$ and hence the latter is a closed subset of $\R$.
\end{remark}

\noindent
Next, we prove a subsolution principle for $(P_\lambda)$:

\begin{lemma}\label{sub}
Let $\lambda>0$, $\underline u\in\w_+\setminus\{0\}$ be a subsolution of $(P_\lambda)$. Then, there exists a solution $u_\lambda\in{\rm int}(\cs_+)$ of $(P_\lambda)$ s.t.\ $u_\lambda\ge\underline u$ in $\Omega$.
\end{lemma}
\begin{proof}
Without loss of generality, we may assume that $\underline u\notin\mathcal{S}(\lambda)$. So set for all $(x,t)\in\Omega\times\R$
\[\hat f_\lambda(x,t) = \begin{cases}
f_\lambda(\underline u(x)) & \text{if $t\le\underline u(x)$} \\
f_\lambda(t) & \text{if $t>\underline u(x)$.}
\end{cases}\]
Clearly, $\hat f_\lambda$ satisfies ${\bf H}_q$. Define $\hat F_\lambda$, $\hat\Phi_\lambda$ as above. Reasoning as in Lemma \ref{exi}, we find $u_\lambda\in\w$ s.t.\
\[\hat\Phi_\lambda(u_\lambda) = \inf_{u\in\w}\,\hat\Phi_\lambda(u).\]
In particular, we have weakly in $\Omega$
\beq\label{sub1}
\fpl u_\lambda = \hat f_\lambda(x,u_\lambda).
\eeq
Testing with $(\underline u-u_\lambda)^+\in\w$, we have
\begin{align*}
\langle\fpl\underline u-\fpl u_\lambda,(\underline u-u_\lambda)^+\rangle &\le \int_\Omega\big[f_\lambda(\underline u)-\hat f_\lambda(x,u_\lambda)\big](\underline u-u_\lambda)^+\,dx \\
&= \int_{\{u_\lambda<\underline u\}}\big[f_\lambda(\underline u)-f_\lambda(\underline u)\big](\underline u-u_\lambda)\,dx = 0.
\end{align*}
By Proposition \ref{mon} \ref{mon2} we have $u_\lambda\ge\underline u$ in $\Omega$. Since $\underline{u}\ge 0$ and $\underline u\ne 0$ by assumption, the same holds true for $u_\lambda$. So, in \eqref{sub1} we can replace $\hat f_\lambda$ with $f_\lambda$. By Proposition \ref{reg} we have $u_\lambda\in C^\alpha_s(\overline\Omega)$. Noting that weakly in $\Omega$
\[\fpl u_\lambda+u_\lambda^{r-1} = \lambda u_\lambda^{q-1} \ge 0,\]
by Proposition \ref{smp} (with $h(t)=(t^+)^{r-1}$) we have $u_\lambda\in{\rm int}(\cs_+)$, in particular $u\in\mathcal{S}(\lambda)$.
\end{proof}

\noindent
Applying Lemma \ref{sub} and the mountain pass theorem, we can now prove multiplicity and comparison of solutions for $(P_\lambda)$ for $\lambda>\lambda_*$:

\begin{lemma}\label{mul}
For all $\lambda>\lambda_*$ there exist two solutions $u_\lambda,v_\lambda\in{\rm int}(\cs_+)$ of $(P_\lambda)$, s.t.\ $u_\lambda-v_\lambda\in{\rm int}(\cs_+)$.
\end{lemma}
\begin{proof}
Let $u_*\in\mathcal{S}(\lambda_*)$ be given by Lemma \ref{str}. We have weakly in $\Omega$
\[\fpl u_* = \lambda_*u_*^{q-1}-u_*^{r-1} < \lambda u_*^{q-1}-u_*^{r-1},\]
hence $u_*$ is a strict subsolution of $(P_\lambda)$. Applying Lemma \ref{sub} with $\underline u=u_*$, we find $u_\lambda\in\mathcal{S}(\lambda)$ s.t.\ $u_\lambda\ge u_*$ in $\Omega$, along with $u_\lambda\neq u_*$. More precisely, since weakly in $\Omega$
\[
\fpl u_*+u_*^{r-1} = \lambda_*u_*^{q-1} < \lambda u_\lambda^{q-1} = \fpl u_\lambda+u_\lambda^{r-1},
\]
by Proposition \ref{scp} (with $h(t)=(t^+)^{r-1}$) we have $u_\lambda-u_*\in{\rm int}(\cs_+)$. In addition, $u_\lambda$ is a global minimizer of the truncated functional $\hat\Phi_\lambda$ in $\w$. Define now
\[U = \big\{u_*+v:\,v\in\w\cap{\rm int}(\cs_+)\big\},\]
a relatively open set in $\cs$ s.t.\ $u_\lambda\in U$. For a.e.\ $x\in\Omega$ and all $t>u_*(x)$ we have
\begin{align*}
\hat F_\lambda(x,t) &= \int_0^{u_*}f_\lambda(u_*)\,d\tau+\int_{u_*}^t f_\lambda(\tau)\,d\tau \\
&= F_\lambda(t)-\big[F_\lambda(u_*)-f_\lambda(u_*)u_*\big].
\end{align*}
Therefore, for all $u\in U$ we have
\begin{align*}
\hat\Phi_\lambda(u) &= \frac{\|u\|_{\w}^p}{p}-\int_\Omega\hat F_\lambda(x,u)\,dx \\
&= \frac{\|u\|_{\w}^p}{p}-\int_\Omega F_\lambda(u)\,dx+\int_\Omega\big[F_\lambda(u_*)-f_\lambda(u_*)u_*\big]\,dx \\
&= \Phi_\lambda(u)-C,
\end{align*}
with $C\in\R$ independent of $u$. So, for all $u\in U$
\[\Phi_\lambda(u) = \hat\Phi_\lambda(u)+C \ge \hat\Phi_\lambda(u_\lambda)+C = \Phi_\lambda(u_\lambda).\]
Therefore, that $u_\lambda$ is a local minimizer of $\Phi_\lambda$ in $\cs$. By Theorem \ref{svh}, it is so in $\w$ as well.
\vskip2pt
\noindent
Set for all $(x,t)\in\Omega\times\R$
\[\tilde f_\lambda(x,t) = \begin{cases}
f_\lambda(t) & \text{if $t<u_\lambda(x)$} \\
\lambda u_\lambda^{q-1}(x)-t^{r-1} & \text{if $t\ge u_\lambda(x)$.}
\end{cases}\]
Clearly, $\tilde f_\lambda:\Omega\times\R\to\R$ satisfies ${\bf H}_q$ and for a.e.\ $x\in\Omega$, $t\in\R$
\[\tilde f_\lambda(x,t) \le f_\lambda(t).\]
Define $\tilde F_\lambda$, $\tilde\Phi_\lambda$ as usual. Integrating both sides of the inequality above, and recalling that both functions vanish at $t\le 0$, we have for a.e.\ $x\in\Omega$ and all $t\in\R$
\[\tilde F_\lambda(x,t) \le F_\lambda(t).\]
This in turn implies for all $u\in\w$
\beq\label{mul1}
\tilde\Phi_\lambda(u) \ge \Phi_\lambda(u).
\eeq
Let $\rho>0$ be s.t.\ $\Phi_\lambda(u)\ge\Phi_\lambda(u_\lambda)$ for all $u\in\w$, $\|u\|_{\w}\le\rho$. Then, for the same $u$'s we have by \eqref{mul1}
\[\tilde\Phi_\lambda(u) \ge \Phi_\lambda(u) \ge \Phi_\lambda(u_\lambda) = \tilde\Phi_\lambda(u_\lambda),\]
i.e., $u_\lambda$ is a local minimizer of $\tilde\Phi_\lambda$.
\vskip2pt
\noindent
We will now prove that $0$ is another local minimizer of $\tilde\Phi_\lambda$. Indeed, set
\[\delta = \lambda^{-\frac{1}{q-r}} > 0.\]
Then, for all $t\in[0,\delta]$
\[f_\lambda(t) = t^{r-1}(\lambda t^{q-r}-1) \le 0,\]
which in turn implies for a.e.\ $x\in\Omega$ and all $t\in[0,\delta]$
\[\tilde F_\lambda(x,t) \le F_\lambda(t) \le 0.\]
There exists $\sigma>0$ s.t.\ for all $u\in\cs$, $\|u\|_{\cs}\le\sigma$ we have $\|u\|_\infty\le\delta$. So, for all $u\in\w\cap\cs$ with $0<\|u\|_{\cs}\le\sigma$ we have
\[\tilde\Phi_\lambda(u) \ge \frac{\|u\|_{\w}^p}{p} > 0 = \tilde\Phi_\lambda(0),\]
hence $0$ is a (strict) local minimizer of $\tilde\Phi_\lambda$ in $\cs$. By Theorem \ref{svh}, it is so in $\w$ as well.
\vskip2pt
\noindent
By \eqref{mul1} and the proof of Lemma \ref{exi}, the functional $\tilde\Phi_\lambda\in C^1(\w)$ is coercive, in particular it satisfies the Palais-Smale condition. In addition, $\tilde\Phi_\lambda$ has two distinct local minima at $u_\lambda$ and $0$, respectively. By the mountain pass theorem (see \cite[Corollary 1]{PS}), there exists a critical point $v_\lambda\in\w\setminus\{u_\lambda,0\}$ of $\tilde\Phi_\lambda$, which satisfies weakly in $\Omega$
\beq\label{mul2}
\fpl v_\lambda = \tilde f_\lambda(x,v_\lambda).
\eeq
Testing \eqref{mul2} with $-v_\lambda^-\in\w$ and applying Proposition \ref{mon} \ref{mon1}, we get
\begin{align*}
\|v_\lambda^-\|_{\w}^p &\le \langle\fpl v_\lambda,-v_\lambda^-\rangle \\
&= \int_\Omega\tilde f_\lambda(x,v_\lambda)(-v_\lambda^-)\,dx \\
&= \int_{\{v_\lambda<0\}}f_\lambda(v_\lambda)v_\lambda\,dx = 0,
\end{align*}
hence $v_\lambda\ge 0$ in $\Omega$. Besides, testing both \eqref{mul2} and $(P_\lambda)$ with $(v_\lambda-u_\lambda)^+\in\w$, we get
\begin{align*}
\langle\fpl v_\lambda-\fpl u_\lambda,(v_\lambda-u_\lambda)^+\rangle &= \int_\Omega\big[\tilde f_\lambda(x,v_\lambda)-f_\lambda(u_\lambda)\big](v_\lambda-u_\lambda)^+\,dx \\
&= \int_{\{v_\lambda>u_\lambda\}}(u_\lambda^{r-1}-v_\lambda^{r-1})(v_\lambda-u_\lambda)\,dx \le 0.
\end{align*}
By Proposition \ref{mon} \ref{mon2} we have $v_\lambda\le u_\lambda$ in $\Omega$. Therefore, in \eqref{mul2} we can replace $\tilde f_\lambda$ with $f_\lambda$. By Proposition \ref{reg} we have $v_\lambda\in C^\alpha_s(\overline\Omega)$, $0\le v_\lambda\le u_\lambda$ in $\Omega$, and $v_\lambda\neq 0,u_\lambda$.
\vskip2pt
\noindent
To conclude, note that weakly in $\Omega$
\[\fpl v_\lambda+v_\lambda^{r-1} = \lambda v_\lambda^{q-1} \ge 0,\]
hence by Proposition \ref{smp} (with $h(t)=(t^+)^{r-1}$) we have $v_\lambda\in{\rm int}(\cs_+)$. Thus, $v_\lambda\in\mathcal{S}(\lambda)$. Also, we have weakly in $\Omega$
\begin{align*}
\fpl v_\lambda+v_\lambda^{r-1} &= \lambda v_\lambda^{q-1} \\
&\le \lambda u_\lambda^{q-1} = \fpl u_\lambda+u_\lambda^{r-1},
\end{align*}
hence by Proposition \ref{scp} (with $h(t)=(t^+)^{r-1}$) we have $u_\lambda-v_\lambda\in{\rm int}(\cs_+)$, which concludes the proof.
\end{proof}

\noindent
In our last lemma, we prove the existence of a biggest solution (in the sense of pointwise ordering) for $(P_\lambda)$, for all $\lambda>\lambda_*$, with strictly increasing dependence on $\lambda$:

\begin{lemma}\label{big}
For all $\lambda\ge\lambda_*$ there exists a biggest solution $\hat u_\lambda\in{\rm int}(\cs_+)$ of $(P_\lambda)$. Also, for all $\lambda>\mu\ge\lambda_*$ we have $\hat u_\lambda-\hat u_\mu\in{\rm int}(\cs_+)$.
\end{lemma}
\begin{proof}
Fix $\lambda\ge\lambda_*$. By Lemmas \ref{str}, \ref{mul} we have $\mathcal{S}(\lambda)\neq\emptyset$. In addition, $\mathcal{S}(\lambda)$ is a compact set in both $\w$ and $\cs$ (see Remark \ref{comp}).
\vskip2pt
\noindent
We prove now that $\mathcal{S}(\lambda)$ is upward directed, in the sense of pointwise ordering. Indeed, let $u,v\in\mathcal{S}(\lambda)$, then by \cite[Lemma 3.1]{FI} the function $\max\{u,v\}\in\widetilde{W}^{s,p}(\Omega)$ is a subsolution of $(P_\lambda)$ (the original result is stated for $p\ge 2$, but the proof also holds in the case $1<p<2$). By Lemma \ref{sub}, then, there exists a solution $w\in\mathcal{S}(\lambda)$ s.t.\ in $\Omega$
\[w \ge \max\{u,v\}.\]
Reasoning as in \cite[Theorem 3.5]{FI} (again, the argument holds for all $p>1$ despite the original statement assumes $p\ge 2$), we deduce the existence of $\hat u_\lambda\in\mathcal{S}(\lambda)$ s.t.\ $\hat u_\lambda\ge u$ in $\Omega$, for all $u\in\mathcal{S}(\lambda)$.
\vskip2pt
\noindent
Now fix $\lambda>\mu\ge\lambda_*$. Let $\hat u_\lambda$, $\hat u_\mu$ be the biggest solutions of $(P_\lambda)$, $(P_\mu)$, respectively. We have weakly in $\Omega$
\begin{align*}
\fpl\hat u_\mu &= \mu\hat u_\mu^{q-1}-\hat u_\mu^{r-1} \\
&< \lambda\hat u_\mu^{q-1}-\hat u_\mu^{r-1},
\end{align*}
i.e., $\hat u_\mu\in{\rm int}(\cs_+)$ is a strict subsolution of $(P_\lambda)$. Reasoning as in Lemmas \ref{sub}, \ref{mul} we find a solution $u_\lambda\in\mathcal{S}(\lambda)$ s.t.\ $u_\lambda-\hat u_\mu\in{\rm int}(\cs_+)$. By maximality of $\hat u_\lambda$, we have in $\Omega$
\[\hat u_\lambda \ge u_\lambda > \hat u_\mu,\]
which implies $\hat u_\lambda-\hat u_\mu\in{\rm int}(\cs_+)$.
\end{proof}

\noindent
We can now prove our main result:
\vskip4pt
\noindent
{\em Proof of Theorem \ref{pqr}.} Let $\lambda_*=\inf\,\Lambda>0$ defined as above. Such definition yields \ref{pqr1}. By Lemma \ref{str} we have \ref{pqr2}. By Lemma \ref{mul}, for all $\lambda>\lambda_*$ there exist two solutions $u_\lambda,v_\lambda$ of $(P_\lambda)$ s.t.\ $u_\lambda-v_\lambda\in{\rm int}(\cs_+)$, and by Lemma \ref{big} we may take $u_\lambda=\hat u_\lambda$, the biggest solution of $(P_\lambda)$, without loss of generality, achieving \ref{pqr3}. Finally, from the second statement of Lemma \ref{big} we have \ref{pqr4}. \qed

\vskip4pt
\noindent
{\bf Acknowledgement.} Both authors are members of GNAMPA (Gruppo Nazionale per l'Analisi Matematica, la Probabilit\`a e le loro Applicazioni) of INdAM (Istituto Nazionale di Alta Matematica 'Francesco Severi') and supported by the the research project {\em Problemi non locali di tipo stazionario ed evolutivo} (GNAMPA, CUP E53C23001670001). S.M.\ is also  partially supported by the projects PIACERI linea 2/3 of the University of Catania and PRIN project 2022ZXZTN2.

\end{document}